\documentclass[10pt,twoside]{amsart}
\usepackage{amsopn}
\usepackage{amsthm}
\usepackage{amssymb,amsmath,amscd, eurosym,mathtools}
\usepackage{comment}
\usepackage{enumitem}
\usepackage[mathscr]{eucal}
\usepackage{tikz}

\usepackage{tikz}
\usepackage{cancel}

\usepackage{pifont}
\usepackage{color}

\usetikzlibrary{decorations.pathreplacing}

\newtheorem{Theorem}{Theorem}[section]
\newtheorem{Lemma}[Theorem]{Lemma}
\newtheorem{Proposition}[Theorem]{Proposition}
\newtheorem{Corollary}[Theorem]{Corollary}
\newtheorem{Conjecture}[Theorem]{Conjecture}

\theoremstyle{Definition}
\newtheorem{Definition}[Theorem]{Definition}  
\theoremstyle{Remark} 
\newtheorem{Remark}[Theorem]{Remark} 
\newtheorem{Example}[Theorem]{Example}
\newtheorem{Acknowledgments}[Theorem]{Acknowledgments}

\newtheorem{Question}[Theorem]{Question}

\DeclareMathOperator{\reg}{reg}

\numberwithin{equation}{section}
\renewcommand{\bar}[1]{\overline{#1}}

\renewcommand{\phi}{\varphi}
\renewcommand{\theta}{\vartheta}
\renewcommand{\epsilon}{\varepsilon}

\begin{document}
\title{Steiner Configurations ideals: containment and colouring}
	
\author[E. Ballico]{Edoardo Ballico}
\address{Dipartimento di Matematica\\
 Via Sommarive, 14 -  38123 Povo (TN), Italy}
\email{edoardo.ballico@unitn.it}
\author[G. Favacchio]{Giuseppe Favacchio}
	\address{Dipartimento di Matematica e Informatica\\
		Viale A. Doria, 6 - 95100 - Catania, Italy}
	\email{favacchio@dmi.unict.it} \urladdr{www.dmi.unict.it/~gfavacchio}
\author[E. Guardo]{Elena Guardo}
\address{Dipartimento di Matematica e Informatica\\
		Viale A. Doria, 6 - 95100 - Catania, Italy}
	\email{guardo@dmi.unict.it} \urladdr{www.dmi.unict.it/guardo}
\author[L. Milazzo]{{Lorenzo Milazzo}{$\dag$}}\thanks{$\dag$   Deceased, March 4, 2019.}
\address{Dipartimento di Matematica e Informatica\\
		Viale A. Doria, 6 - 95100 - Catania, Italy}
	\email{milazzo@dmi.unict.it} \urladdr{www.dmi.unict.it/milazzo}
\author[A. Thomas]{Abu Chackalamannil Thomas}
\address{Department of Mathematics,\\
		Tulane University, New Orleans, LA, U.S.A 70118}
	\email{athoma17@tulane.edu} 

\thanks{Last updated: Jan 15, 2021}

%\begin{document}
%\maketitle


\begin{abstract}
Given a homogeneous ideal $I \subseteq k[x_0,\dots,x_n]$,  the Containment problem  studies the relation between symbolic and regular powers of $I$, that is, it asks for which pair $m, r \in \mathbb{N}$,  $I^{(m)} \subseteq I^r$ holds. In the last years, several conjectures have been posed on this problem, creating an active area of current interests and ongoing investigations. In this paper, we investigated  the Stable Harbourne Conjecture and the Stable Harbourne -- Huneke Conjecture and we show that they hold for the defining ideal of a Complement of a Steiner configuration of points in $\mathbb{P}^{n}_{k}$.  We can also show that the ideal of a Complement of a Steiner Configuration of points has expected resurgence, that is, its resurgence is strictly less than its big height, and it also satisfies Chudnovsky and Demailly's Conjectures. Moreover, given a hypergraph $H$, we also study the relation between  its colourability  and the failure of the containment problem for the cover ideal associated to $H$. We apply these results in the case that $H$ is a Steiner System.	
\end{abstract}

\maketitle
%%%%%%%%%%%%%%%%%%%%%%%%%%%%%%%%%%%%%%%%%%%%%%%%%%%%%%%%%%%%%%%%%%%%%%%%
%%   Introduction                                                     %%
%%%%%%%%%%%%%%%%%%%%%%%%%%%%%%%%%%%%%%%%%%%%%%%%%%%%%%%%%%%%%%%%%%%%%%%%

\section{Introduction}
In this paper we continue the study of Steiner configurations of points and their invariants, such as Hilbert Function, Betti numbers, Waldschmidt constant, regularity, resurgence found in \cite{BFGM}. We will focus on the Containment problem and we will show that the Stable Harbourne Conjecture and the Stable Harbourne -- Huneke Conjecture hold for the defining ideal of a Complement of a Steiner configuration of points in $\mathbb{P}^n_k:=\mathbb{P}^n$. As pointed out in Remarks 2.5 and 2.6 in \cite{BFGM} in the language of Algebraic Geometry/Commutative Algebra, Steiner configurations of points and their Complement are special subsets of star configurations.

First, we give an overview on the Containment problem to introduce the related conjectures.  Then we devote Section 2 to recall notation, definitions and known results for a Steiner configuration of points and its Complement that we will use to prove the results of this paper.
Let $I$ be a homogeneous ideal in the standard graded polynomial ring $R:=k[x_0, \ldots , x_n ]$, where $k$ is a field.  Given an integer $m$, we denote by $I^m$ the regular power of the ideal $I$. The $m$-th \textit{symbolic power}  of $I$ is defined as 
$$I^{(m)}=\bigcap_{\mathfrak p\in Ass(I)} (I^m R_{\mathfrak p}\cap R) $$
where $Ass(I)$ denotes the set of associated primes of $I$ and $R_{\mathfrak p}$ is the localization of $R$ at a prime ideal ${\mathfrak p}$.

If $I$ is a radical ideal (this includes for instance squarefree monomial ideals and ideals of finite sets of points) then 
$$I^{(m)}=\bigcap_{\mathfrak p\in Ass(I)} {\mathfrak p}^m.$$

Symbolic powers of ideals play a significant role in the famous Zariski-Nagata Theorem (see \cite{N,Z}). If $R$ is a polynomial ring over an algebraically closed field $k$, then $I^{(m)}$ consists precisely of those functions which vanish on the algebraic variety defined by $I$ with multiplicity at least $m$. It is easy to show from the definition that $I^{r}\subseteq I^{(m)}$ if and only if $r\ge m$. The reverse inclusion  $I^{(m)} \subseteq I^r$ motivates the following question.

\begin{Question} (\textit{Containment problem}) \label{q.containement}
Given a homogeneous ideal $I \subseteq k[x_0,\dots,x_n]$, for which pairs $m, r \in \mathbb{N}$, does $I^{(m)} \subseteq I^r$ hold?
\end{Question}

One of the initial works that introduce Question \ref{q.containement} is \cite{S}. The problem is still open in general and in the last couple of decades it was extensively studied  for several classes of ideals, in particular  for ideals  defining finite sets of points in projective and multiprojective spaces, see \cite{BH2,BH, CEH, CHHVT1, ELS,FG, FG2, GHaVT, GHaVT1,GHaVT2013a, MS} just to cite some among all the known results. Containment problems are useful in giving lower bounds to nonzero homogeneous forms vanishing through a finite set of points with a fixed multiplicity. 

It is of great interest to study the ideals of fat points. Given distinct points $P_1,\dots,P_s\in \mathbb{P}^n$ and nonnegative integers $m_i$ (not all $0$), let $Z=m_1p_1+\dots+m_sp_s$ denote the scheme (called a fat point scheme) defined by the ideal $I_Z =\cap_{i=1}^{s}(I_{P_i}^{m_i})\subseteq k[\mathbb{P}^n]$, where $I_{P_i}$ is the ideal generated by all homogeneous polynomials vanishing at $P_i$.  Symbolic powers of $I_Z$ take the form $I_Z^{(m)}=I_{mZ}= \cap_{i=1}^{s}I_{P_i}^{mm_i}$. We say that $Z$ is reduced if $I_Z$ is a radical ideal.

The Containment problem also helps us to bound certain useful invariants like Waldschmidt constant, $\widehat \alpha(I)$, of an ideal $I$ defined as $$\widehat \alpha (I) = \lim_{m\to \infty} \frac{\alpha(I^{(m)})}{m},$$ \noindent where   $\alpha(I)$ is the minimum integer $d$ such that $I_d\neq (0)$, that is, it is the least degree of a minimal generator of $I$. This limit exists and was first defined by Waldschmidt \cite{W} for ideals of finite sets of points in the context of complex analysis; specifically, in our  language, the problem was to determine the minimal degree of a hypersurface that passed through
a collection of points with prescribed multiplicities.

The following slight different version of Question \ref{q.containement} was introduced in \cite{HaHu}.
 Recall that the {\it big height} of an ideal $I$ refers to the maximum of all the heights of its associated prime ideals.

\begin{Conjecture}
Let $Z\subset \mathbb{P}^n$ be a fat point scheme and $I:=I_Z$ the ideal defining $Z$. Let $\mathcal{M} = (x_0,\dots,x_n)$ be the graded maximal ideal. Then $I^{(rn)}\subseteq \mathcal{M}^{r(n-1)}I^r$ holds for all $r >0$.
\end{Conjecture}

B. Harbourne conjectured in \cite{BDHKaKnSySz}:

\begin{Conjecture}
Given a nonzero, proper, homogeneous, radical ideal $I \subseteq k[x_0,\dots,x_n]$ with big height $h$, $$I^{(hr-h+1)} \subseteq I^r$$ for all $r\geq 1.$
\end{Conjecture} 

A counterexample to the above conjecture was initially found in \cite{DST}. 

A celebrated result of \cite{ELS, HoHu, MS} is shown in the next theorem.
\begin{Theorem}\label{einHH}
Let $R$ be a regular ring and $I$ a radical ideal in $R$. Then for all $n \in \mathbb{N}$,
$$I^{(hn)} \subseteq I^n,$$  
\noindent whenever $h$ is the big height of $I$. 
\end{Theorem}One could hope to sharpen the containment by reducing the symbolic power on the left hand side by a constant or increasing the ordinary power on the right hand side by a fixed constant. This motivates us to look at stable versions of Conjectures 2.1 and 4.1 in \cite{HaHu}, respectively.

\begin{Conjecture} (Stable Harbourne Conjecture)\label{SHC}
Given a nonzero, proper, homogeneous, radical ideal $I \subseteq k[x_0,\dots,x_n]$ with big height $h$, then $$I^{(hr-h+1)} \subseteq I^r$$ for all $r\gg 0.$
\end{Conjecture}

\begin{Conjecture}(Stable Harbourne--Huneke Conjecture) \label{SHHC}
Let $I \subseteq k[x_0,\dots,x_n]$ be a homogeneous radical ideal of big height $h$. Let $\mathcal{M} = (x_0,\dots,x_n)$ be the graded maximal ideal. Then for~$r\gg0$,

\begin{enumerate}
    \item $I^{(hr)} \subseteq \mathcal{M}^{r(h-1)}I^r$
    \item $I^{(hr-h+1)} \subseteq \mathcal{M}^{(r-1)(h-1)}I^r$.
\end{enumerate}

\end{Conjecture}

In the study of finding the least degree of minimal generators of an ideal $I$, Chudnovsky made the following:

\begin{Conjecture} (Chudnovsky's Conjecture). Suppose that $k$ is an algebraically closed field of characteristic $0$. Let $I$ be the defining ideal of a set of points $X\subseteq \mathbb{P}^n_k$. Then, for all $h > 1$,
$$\frac{\alpha(I^{(h)})}{h}\geq \frac{\alpha(I)+n-1}{n}.$$
\end{Conjecture}

A generalization of Chudnovsky’s Conjecture is the following:
\begin{Conjecture} (Demailly's Conjecture). Suppose that $k$ is an algebraically closed field of characteristic
$0$. Let $I$ be the defining ideal of a set of points $X\subseteq \mathbb{P}^n_k$ and let $m \in \mathbb{N}$ be any integer. Then,
for all $h > 1$,
$$\frac{\alpha(I^{(h)})}{h}\geq \frac{\alpha(I^{(m)})+n-1}{m+n-1}.$$

\end{Conjecture}
Two recent  preprints, \cite{BiGrTN,BiGrTN2}, focus on the Containment problem and related conjectures. In the first one, the authors show that Chudnovsky’s Conjecture holds for sufficiently many general points and to prove it they show that one of the containments conjectured by Harbourne and Huneke holds eventually, meaning for large powers (see Theorem 4.6 in \cite{BiGrTN}). They also show other related results, for example, that general sets of points have expected resurgence and thus satisfy the Stable Harbourne Conjecture.

In the second preprint, the authors show that Demailly's Conjecture (which is a generalization of Chudnovsky's) also holds for sufficiently many general points, for star configurations (in general, not just points) and for generic determinantal ideals.

In this paper we prove that the Stable Harbourne Conjecture and the Stable Harbourne--Huneke Conjecture hold for ideals defining the Complement of a Steiner Configuration of points in $\mathbb{P}^n$ that are special subsets of star configurations and, hence, far from being general. We will give more details in Section \ref{s. Harbourne conj}.

We remark that the least degree of a minimal generator of the ideal defining the Complement of a Steiner Configuration of points in $\mathbb{P}^n$ is strictly less than the least degree of a minimal generator of the ideal of a star configurations (see Theorem \ref{p. alpha(J_T^m)} and also  Proposition 2.9 in \cite{GHM}). So, it is worth investigating whether the Containment problem and its related conjectures hold for the Complement of a Steiner Configuration of points in $\mathbb{P}^n$.

In \cite{BFGM} the authors constructed a squarefree monomial ideal $J$ associated to a set $X$ of points in $\mathbb{P}^n$ constructed from the Complement of a Steiner system. The ideal $I_X$ defining the Complement of a Steiner system is not a monomial ideal. But the authors proved that the symbolic powers of $I_X$ and $J$ share the same homological invariants (see Proposition 3.6 in \cite{BFGM}). This was possible because $J$ is the Stanley-Reisner ideal of a matroid, so its symbolic powers define an arithmetically Cohen-Macaulay (ACM for short) scheme which gives, after proper hyperplane sections, the scheme of fat points supported on $X$.  But we point out that the {\em regular} powers of $J$ are not necessarily ACM  anymore and we cannot relate them to squarefree monomial ideals. Thus, the homological invariants of the regular powers of $J$ are not immediately correlated to that of $I_X$.

In \cite{CEH} the authors proved that the Chudnovsky’s Conjecture, the Harbourne’s Conjecture and the Harbourne--Huneke containment conjectures hold for squarefree monomial ideals. 

As previously remarked, since the ideal $I_X$ defining the Complement of a Steiner system is not a squarefree monomial ideal, we cannot recover the Stable Harbourne Conjecture and the Stable Harbourne -- Huneke Conjecture  using \cite{CEH}.

We also point out that the two above preprints \cite{BiGrTN,BiGrTN2} do not compute the Waldschmidt constant exactly for any class of ideals, they study lower bounds for the Waldschmidt constant. And, since in \cite{BFGM} the authors found the exact value of the Waldschmidt constant for the Complement of a Steiner configurations of points, then Chudnovsky and Demailly’s Conjectures easily follow for our class of ideals (see Section \ref{s. Harbourne conj}).

For other results on this topic we can also see \cite{FMX, Mantero}.

Another tool useful to measure the non--containment among symbolic and ordinary powers of ideals is the notion of {\it resurgence} $\rho(I)$ of an ideal $I$,  introduced in \cite{BH} that gives some notion of how small the ratio $m/r$ can be and still be sure to have $I^{(m)}\subseteq I^r$.
\begin{Definition}
	Let $I$ be a nonzero, proper ideal in a commutative ring $R$, the \textit{resurgence} of the ideal $I$ is given by $$\rho(I) = \sup\left\{\frac{m}{r} \quad \vert \quad I^{(m)} \nsubseteq I^r\right\}.$$
\end{Definition} 

It always satisfies $\rho(I) \geq 1$. The groundbreaking results of \cite{ELS,HoHu,MS} show that $\rho(I) \leq h$, where $h$ is the big height of the radical ideal $I$. This motivates us to ask whether $\rho(I)$ can strictly be less than its big height and which are some of the interesting consequences. Although there are few cases where the resurgence has been computed, in general, it is extremely difficult to estimate the exact value for $\rho(I)$. The reader can look at \cite{DHN+} for the first examples where the resurgence and the asymptotic resurgence are not equal. An asymptotic  version of the resurgence was introduced  in the paper \cite{GHaVT}. 
\begin{Definition}
	For a nonzero, proper homogeneous ideal $I \subseteq k[x_0,\dots,x_n]$, the \textit{asymptotic resurgence} $\rho_a(I)$ is defined as follows:$$\rho_a(I) = \sup\left\{\frac{m}{r}\quad \vert \quad I^{(mt)} \nsubseteq I^{rt}, \quad \text{for all} \quad t\gg0\right\}.$$
\end{Definition}

It is clear from the definition that $1 \leq \rho_a(I) \leq \rho(I)$. 
As pointed out in  \cite{HKZ},  DiPasquale,  Francisco,  Mermin, Schweig showed that $\rho_a(I) = sup\{m/r:~I^{(m)}\nsubseteq \overline{I^r}\}$, where $ \overline{I^r}$ is the integral closure of $I^r$ (see also \cite{DiPFMS} Corollary 4.14) .

%In this paper we give a criteria to compute the resurgence for the Complement of a Steiner configurations of points. 
In this paper we study the containment properties of the ideal defining a Complement of a Steiner configuration of points  in $\mathbb{P}^n$.
Section \ref{Prelim} is devoted to recall notation, definitions and known results from \cite{BFGM} that we will use in the next sections.
The main result of Section \ref{s. Harbourne conj} is Theorem \ref{t. Stable main} where we prove that an ideal defining the Complement of a Steiner Configuration of points in $\mathbb{P}^n$ satisfies both the Stable Harbourne Conjecture and the Stable Harbourne--Huneke Conjecture. In Lemma \ref{l. main lemma}, we give a criterion for when the resurgence number can be computed in finite number of steps. This result improves the bounds found in Corollary 4.8 in \cite{BFGM}. We also point out that  Lemma \ref{l. main lemma} is similar to results from \cite{DD, DiPFMS}.  As a consequence, in Corollary \ref{expected} we  show that the ideal of a Complement of a Steiner Configuration of points has expected resurgence, that is, its resurgence is strictly less than its big height (see \cite{GrHuM} for the first definition). Moreover, using Theorem \ref{p. alpha(J_T^m)}, Corollaries \ref{coro} and \ref{coro2}, we show that the ideal of a Complement of a Steiner Configuration of points satisfies Chudnovsky and Demailly's Conjectures (see Corollary \ref{coro2}, Corollary \ref{Chudnovky} and Corollary \ref{demailly}).

Finally, in Section \ref{s.coloring}, given a hypergraph ${H}$, we also study the relation between  its colourability  and the failure of the Containment problem for the cover ideal associated to $H$.  The ideas come from the paper \cite{FHVT2} where the authors start to study the natural one-to-one correspondence between squarefree monomial ideals and finite simple hypergraphs via the cover ideal construction. 

There exists an extensive literature on the subject of colourings both from Design Theory and Algebraic Geometry/Commutative Algebra point of view. Among all, we make use of  \cite{BTV, FHVT2, HVT, GMV} as some referring texts for preliminaries on hypergraph theory and associated primes and  for an algebraic method to compute the chromatic number, respectively. 

Most of the existing papers are devoted to the case of weak
colourings (or vertex colourings), i.e. colourings where the colours are assigned to the elements in such a way that no hyperedge is monochromatic (i.e. no hyperedge has all its elements assigned the same colour). The reader can see \cite{BTV} or  Chapter 3 in \cite{GMV} for other different types of colouring a hypergraph, such as {\em strong vertex colouring, vertex equicolouring, good colouring of $H$}.

In this paper we use the case of  weak colouring to get result on the Containment problem since it is the one commonly used in Combinatorial Commutative Algebra.
The main result of this section is Theorem \ref{t.containement fails} that more generally predicts the failure of the containement for squarefree monomial ideals based on the definition of {\it coverability} (see Definition \ref{d. d-cover}). We apply these results in the case that $H$ is a Steiner System.

We end the paper recalling some open questions posed in \cite{BFGM} and that are still under investigations and posing  new ones as possible further research problems.

%%%%%%%%%%%%%%%%%%%%%%%%%%%%%%%%%%%%%%%%%%
\section{Notation, definitions and known results for Ideals of a Steiner configuration of points and its Complement}\label{Prelim}
%\subsection{Ideals of  Steiner configuration of points and its Complement}\label{ss.points}

In this section  we recall  the main results from \cite{BFGM}, where the authors studied the  homological properties of ideals constructed from Steiner systems, especially in the zero dimensional case of $\mathbb P^n$.

A Steiner system $(V,B)$ of type  $S(t,n,v)$  is %a $t-(v,n,1)$ design, that is,  
a collection $B$ of $n$-subsets (blocks) of a $v$-set $V$ such that each $t$-tuple of  $V$ is contained in a unique block in $B$.  The elements in $V$ are called vertices or points and those of $B$ are called blocks.  In particular,  a Steiner triple system of order $v$, $STS(v)$, is a collection of triples ($3$-subsets) of $V$, such that each unordered pair of elements is contained in precisely one block, and a Steiner quadruple system of order $v$, $SQS(v)$, is a collection of quadruples (4-subsets) of $V$ such that each triple is found in precisely one block. 

The existence of a Steiner system strongly depends on the parameters $(t,n,v)$. If a Steiner system $(V,B)$ of type $S(t,n,v)$ exists, then \[|B|=\frac{{\binom{v}{t}}}{{\binom{n}{t}}}.\]

We use \cite{Colbourn} and \cite{ColbournRosa} as main references for all the background on design theory.

We recall the most known example.
\begin{Example}\label{e.STS(7)_1}
		  One of the simplest and most known example of Steiner system is the Fano Plane. It is unique up to isomorphism and it is a  Steiner system $S(2,3,7)$ with block set
		$$B:=\{\{ 1,2,3\}, \{3,4,5\},\{3,6,7\},\{1,4,7\},\{2,4,6\},\{2,5,7\},\{1,5,6\}\}.$$
		\end{Example}
		
For the ease of the reader, we recall some definitions and results from \cite{BFGM}.

Let  $V:=\{1, \ldots, v\}$  and $\mathcal H:=\{H_1, \ldots H_v\}$ be a collection of distinct hyperplanes  of $\mathbb P^n$, where $n\leq v$. Say $H_j$ defined by the linear forms $\ell_j$  for $j=1, \dots, n$. Assume that any $n$ hyperplanes in $\mathcal H$ meet properly, i.e., they meet in a point. 
There is a natural way to associate a point in $\mathbb P^n$ to a subset of $n$ elements of $V$.  For  $\sigma:=\{\sigma_1, \ldots, \sigma_n\}\subseteq V$, we denote by $P_{\mathcal H, \sigma}$  the point obtained by intersecting the hyperplanes $H_{\sigma_1},\ldots, H_{\sigma_n}.$ Then, the ideal $I_{P_{\mathcal H, \sigma}}=(\ell_{\sigma_1},\ldots, \ell_{\sigma_n})\subseteq k[\mathbb P^n]$ is  the vanishing ideal of the point $P_{\mathcal H, \sigma}$.

\begin{Definition}\label{Steinergeneral}Let $Y$ be a collection of subsets of $V$ containing $n$ elements, and $\mathcal H$ a set of hyperplanes meeting properly. We define the following set  of points in $\mathbb P^n$  with respect to $\mathcal H$  
	$$X_{\mathcal H,Y}:=\bigcup_{\sigma\in Y}~ P_{\mathcal H, \sigma}$$ and  its defining ideal $$I_{X_{\mathcal H,Y}}:=\bigcap_{\sigma\in Y} ~  I_{P_{\mathcal H, \sigma}}.$$
\end{Definition}

Denoted by $C_{(n,v)}$  the set containing all the subsets of $V$ with $n$ elements the above definition applied to a Steiner system gives us two different sets of points. 

\begin{Definition}\label{Steiner}  Let $(V,B)$ be a Steiner system of type $S(t,n,v)$ with $t<v\leq n$. We associate to $B$ the following set  of points in $\mathbb P^n$   $$X_{\mathcal H,B}:=\bigcup_{\sigma\in B } ~ P_{\mathcal H, \sigma}$$ and  its defining ideal $$I_{X_{\mathcal H,B}}:=\bigcap_{\sigma\in B} ~ I_{P_{\mathcal H, \sigma}}.$$
 We call $X_{\mathcal H, B}$ the  Steiner configuration of points associated to the Steiner system $(V,B)$ of type $S(t,n,v)$ with respect to $\mathcal H$ (or  just $X_{ B}$  if there is no ambiguity).
\end{Definition}

\begin{Definition}\label{C-Steiner}
 Let $(V,B)$ be a Steiner system of type $S(t,n,v)$  with $t<n\leq v$. We associate to $C_{(n,v)}\setminus B$ the following set  of points in $\mathbb P^n$   $$X_{\mathcal H,C_{(n,v)}\setminus B}:=\bigcup_{\sigma\in C_{(n,v)}\setminus B}~ P_{\mathcal H, \sigma}$$ and  its defining ideal $$I_{X_{\mathcal H,C_{(n,v)}\setminus B}}:=\bigcap_{\sigma\in C_{(n,v)}\setminus B} ~ I_{P_{\mathcal H, \sigma}}.$$
We call $X_{\mathcal H,C_{(v,n)}\setminus B}$ the  Complement of a Steiner configuration of points  with respect to $\mathcal H $  (or C-Steiner $X_{\text{C}}$ if there is no ambiguity).
\end{Definition}
%\begin{remark}\label{stars}
As pointed out in \cite{BFGM},  Remarks 2.5 and 2.6  a Steiner configuration of points and its Complement are subschemes of a star configuration of ${\binom{v}{n}}$  points in $\mathbb{P}^n$ (see \cite{CCGVT, CGVT,GHM, GHMN} just to cite some reference on star configurations).

We also have
$$\deg X_{\mathcal H, C_{(n,v)}\setminus B}= {\binom{v}{t}}- |B|= {\binom{v}{t}}- \dfrac{{\binom{v}{t}}}{{\binom{n}{t}}}.$$

We recall the most known construction of Steiner Configuration of points and its Complement.

\begin{Example}\label{e.STS(7)}
		Consider  the Steiner configuration associated to $(V,B)$ of type $S(2,3,7)$  as in  Example \ref{e.STS(7)_1}. Take $\mathcal H :=\{H_1, \ldots, H_7 \}$  a collection of $7$ distinct hyperplanes $H_i$ in  $\mathbb P^3$   defined by a linear form $\ell_i$ for $i=1,\dots,7$, respectively, with the property that any $3$ of them meet in a point $P_{\mathcal H,\sigma}=H_{\sigma_1}\cap H_{\sigma_2}\cap H_{\sigma_3}$, where $\sigma=\{\sigma_1,\sigma_2,\sigma_3\}\in B$. We get  that $X_{\mathcal H,C_{(3,7)}}$ is a  {\em star configuration} of ${\binom{7}{3}}=35$ points in $\mathbb{P}^3$,  $X_{\mathcal H,B}:=\cup_{\sigma \in B}~ \{P_{\mathcal H, \sigma} \}$ is  a Steiner configuration consisting of $7$ points in $\mathbb P^3$ and  $X_{\mathcal H, C_{(3,7)}\setminus B}$ is  a C-Steiner configuration  consisting of ${\binom{7}{3}}-7= 28$ points in $\mathbb P^3.$  Their defining ideals are respectively, $$I_{X_{\mathcal H,B}}:=\cap_{\sigma\in B} ~ I_{P_{\mathcal H, \sigma}} \textrm{ and } I_{X_{\mathcal H,C_{(3,7)}\setminus B}}:=\cap_{\sigma\in C_{(3,7)}\setminus B} ~ I_{P_{\mathcal H, \sigma}}.$$
	\end{Example} 

In \cite{BFGM} the authors constructed a squarefree monomial ideal $J$ associated to a set $X_{\mathcal H, C}$ of points in $\mathbb{P}^n$ constructed from the Complement of a Steiner system. The ideal $I_{X_{\mathcal H, C}}$ defining the Complement of a Steiner system is not a monomial ideal. But the authors proved that the symbolic powers of $I_{X_{\mathcal H, C}}$ and $J$ share the same homological invariants (see Proposition 3.6 in \cite{BFGM}). 

The following results give the least degree of a minimal generator and the regularity and the Waldschmidt constant of an ideal defining the Complement of a Steiner configuration of points, respectively.

\begin{Theorem}[\cite{BFGM}, Theorem 3.9]\label{p. alpha(J_T^m)}
	Let $(V,B)$ be  a Steiner system of type $S(t,n,v)$.  Set $I_{X_{\mathcal H,C}}:=I_{X_C}$ the ideal defining the Complement of the Steiner configuration of points associated to $S(t,n,v)$. Then
	\begin{itemize}
		\item[i)] $\alpha(I_{X_{C}})=v-n$;
		\item[ii)] $\alpha(I_{X_{C}}^{(q)})=v-n+q$ for $2\le q<n$;
		\item[iii)] $\alpha(I_{X_{C}}^{(m)})=\alpha(I_{X_{C}}^{(q)})+pv$, where $m=pn+q$ and $0\le q<n$ and $\alpha(I_{X_{C}}^{(n)}) =\alpha(I_{X_{C}}^{(0)}) +v=v$.  
	\end{itemize} 
\end{Theorem}

\begin{Corollary}[\cite{BFGM}, Corollary 4.2]\label{coro}
Let $\text{reg}(I_{X_C})$ be the regularity of a  Complement of a Steiner configuration. Then $\text{reg}(I_{X_C})=\alpha(I_{X_C})+1=v-n+1$.
\end{Corollary}

\begin{Corollary}[\cite{BFGM}, Corollary 3.12]\label{coro2}
If $(V, B)$ is a Steiner system of type $S(t, n, v)$, then the Waldschmidt constant of its Complement is $\widehat{\alpha}(I)=\frac{v}{n}$.
\end{Corollary}

\section{ Asymptotic resurgence and Stable Harbourne Conjecture}\label{s. Harbourne conj}

Containment problems have been of interest among commutative algebraists and algebraic geometers. In the last decade, several conjectures related to this problem have been posed creating an active area of  current interests and outgoing investigations.

A celebrated result of \cite{ELS, HoHu, MS} is that $I^{(hn)} \subseteq I^n$ for all $n \in \mathbb{N}$, whenever $h$ is the big height of $I$. One could hope to sharpen the containment by reducing the symbolic power on the left hand side by a constant or increasing the ordinary power on the right hand side by a fixed constant. This motivates us to look at the Stable Harbourne Conjecture and the Stable Harbourne -- Huneke Conjecture and study which class of ideals satisfies them. Here, we prove that the ideal defining a Complement of a Steiner configurations of points satisfies both conjectures. We need to recall some known results.

In \cite{Gr}, the Conjecture \ref{SHC} is shown to hold

\begin{enumerate}
    \item if there exists $k > 0$ such that $I^{(hk-h)} \subseteq I^k$;
    \item if  $I^{(hk-h+1)}\subseteq I^k$ for some $k$ and $I^{(r+h)}\subseteq I I^{(r)}$ for all $r\geq k$;
    \item if the resurgence satisfies $\rho(I) < h$.
\end{enumerate}

In particular, condition (2) gives a criterion for the Stable Harbourne Conjecture (SHC for short) to hold. Namely, for a radical ideal of big height $h$, if for all $k\geq 1$, it is  $I^{(k+h)} \subseteq II^{(k)}$ and fix an integer $C$ and $m$ such that $I^{(hm - C)} \subseteq I^m$ holds, then for all $q \geq m$, we have
 $$I^{(hq-C)} = I^{(h(q-m)+ hm-h+h-C)} \subseteq II^{(h(q-m-1)+hm-C)} \subseteq I^{q-m}I^{(hm-C)} \subseteq I^{q-m}I^m = I^q,$$ 
\noindent that is,  $I^{(hq-C)} \subseteq I^q$.

\begin{Theorem}[Theorem 2.5, \cite{Gr}]\label{eloisa}
Let $R$ be a regular ring containing a field, and let $I$ be a radical ideal in $R$ with big height $h$. If  $I^{(h(m-1))} \subseteq I^{m}$ \noindent for some $m\geq 2$, then $I^{(h(k-1))} \subseteq I^{k}$ \noindent for all $k>>0$ (indeed for all $k\geq hm$).
\end{Theorem}

We  have learned that Harbourne, Kettinger, Zimmitti in \cite{HKZ} and DiPasquale, Drabkin in \cite{DD} proved independently that $\rho_a(I)< h$ if and only if $\rho(I)< h$.
As pointed out in \cite{DD} Remark 2.3, the next result is similar as Proposition 4.1.3 of Denkert's thesis \cite{D}, as  Lemma 4.12 in  \cite{DiPFMS}  and as Proposition 2.2 in  \cite{DD}.

For the ease of the reader, we adapt the proof in our case.

\begin{Lemma}\label{l. main lemma}
Let $I \subseteq k[x_0,\dots,x_n]$ be a homogeneous radical ideal with $big\ height (I) = h$, such that $\rho(I) > \rho_a(I) $. Suppose we have the equality $$\rho_a(I)  = \frac{hr_1-h}{r_1}$$ for some $r_1 > 0$. Then $\rho(I)$ can be computed by taking the maximum of finitely many $\frac{s}{r}$ with~$I^{(s)} \nsubseteq I^r$.
\end{Lemma}

\begin{proof}

Using Briancon Skoda Theorem (\cite{SH}, Corollary 13.3.4), we have that $\overline{I^{r+n}} \subseteq I^r$, where $n + 1$ is the number of variables in the polynomial ring and $\bar{I}$ denotes the integral closure of the ideal $I$.
For $s, r \in \mathbb{N}$ such that $ I^{(s)} \nsubseteq I^r$ then $I^{(s)} \nsubseteq \overline{I^{r+n}}$. Using \cite{DiPFMS}, Lemma 4.12,  we get $\frac{s}{r+n} < h(1-1/r_1) = \rho_a(I)$, that is,
$$\frac{s}{r} < (1+n/r)h(1-1/r_1).$$

If $\rho(I) > \rho_a(I)$, applying \cite{DD} Proposition 2.2, then there exist $s_0,r_0$, such that 
$I^{(s_0)}\nsubseteq I^{r_0}$ and 
$$ \rho(I)\geq \frac{s_0}{r_0} \geq (1 + \frac{n}{r})h(1-\frac{1}{r_1}),$$

\noindent
solving for $r$ gives us the inequality $$r \geq \frac{n}{\frac{s_0/r_0}{h(1-1/r_1)}-1},$$
\noindent 
so whenever $r \geq \frac{n}{\frac{s_0/r_0}{h(1-1/r_1)}-1}$ and $s$ 
is such that $I^{(s)} \nsubseteq I^r$, we have $\frac{s}{r} < \frac{s_0}{r_0}$. 

Hence, it suffices to look at $$r \leq \frac{n}{\frac{s_0/r_0}{h(1-1/r_1)}-1}$$ and 
$s \leq (r+n)h(1-\frac{1}{r_1})$.
\end{proof}

\begin{Corollary}\label{expres}
If the resurgence can be computed by taking the maximum of finitely many ratios of the form $\frac{m}{r}$ for which $I^{(m)} \nsubseteq I^r$, then $\rho(I) < h$. 
\end{Corollary}
\begin{proof} Suppose we have $\frac{a}{b} = h$, then $I^{(hb)} \nsubseteq I^b$ is a contradiction since $I^{(hk)}\subseteq I^k$ for all $k\in \mathbb N$ (from \cite{ELS, HoHu, MS}). Hence, $\rho(I) < \frac{a}{b} = h$.
\end{proof}
The next proposition shows that Conjecture 3.1 in \cite{BCH} holds for  the Complement of a Steiner Configuration of points.

\begin{Proposition}\label{prop}
Let $I \subset k[\mathbb{P}^n]$ be an ideal defining a Complement of a Steiner Configuration of points and let $\mathcal{M} = (x_0,\dots,x_n)$ be the homogeneous maximal ideal. Then  $I^{(nr)} \subseteq \mathcal{M}^{rn}I^r$ holds for all $ r \in \mathbb{N}.$
\end{Proposition}
\begin{proof}
From Theorem \ref{p. alpha(J_T^m)}, we have $\alpha(I^{(nr)}) = rv$. From \cite{BFGM}, Corollary 4.7, we have $\omega(I^r) = \alpha(I^r) = r(v-n)$, where $\omega(I)$ is the maximum of the generating degrees of the ideal $I$. Since $I^{(nr)}\subseteq I^r$ for all $r \geq 1$, we have $\alpha(I^{(nr)}) \geq r\omega(I) = r\alpha(I)$ and $\alpha(I^r) = r(v-n)$, so $\alpha(I^{(nr)}) - r\omega(I)$ = $rv - r(v-n) = rn$. Since every minimal generator of $I^{(nr)}$ is contained inside $I^r$ and the difference between the degree of any nonzero homogeneous polynomial in $I^{(nr)}$ and that of generators of $I^r$ is at least $rn$, we have that $I^{(nr)} \subseteq \mathcal{M}^{rn}I^r$, the conclusion follows.
\end{proof}

We prove the main result of this section:

\begin{Theorem}\label{t. Stable main}
    Let $I \subseteq k[x_0,\dots,x_n]$ be the ideal defining the Complement of a Steiner Configuration of points in $\mathbb{P}^n_k$. Then $I$ satisfies 
    
    \begin{enumerate}
        \item Stable Harbourne--Huneke  Conjecture;
        \item Stable Harbourne Conjecture.
    \end{enumerate}
    
\end{Theorem}
\begin{proof}
(1) Consider the Steiner Configuration of points $S(t,n,v)$ in $\mathbb{P}_k^n$ and $I:=I_C$ the ideal defining its Complement.

Using  Theorem \ref{p. alpha(J_T^m)} , (iii), it is  $\alpha(I^{(n(r-1))}) = (r-1)v$.  Using Corollary \ref{coro}, and choosing $r\gg0$, such that 
$$(r-1)v \geq r\cdot\reg(I) = r(v-n+1)$$ we get $I^{(n(r-1))} \subseteq I^r.$
Moreover, since $\alpha(I^{(n(r-1))}) - \alpha(I^r) = v(r-1) - r(v-n) = rn-v$, we get $I^{(n(r-1))} \subseteq \mathcal{M}^{rn-v}I^r.$
Using Euler's Formula, we get $$\begin{array}{rcl}
I^{(n(r-1)+1)} &\subseteq& \mathcal{M}I^{(n(r-1))}\subseteq \mathcal{M}^{rn-v+1}I^r= \mathcal{M}^{rn-v-n+n+1}I^r = \mathcal{M}^{rn-n-(v-n)+1}I^r\\
&\subseteq& \mathcal{M}^{rn-n-r+1}I^r = \mathcal{M}^{(r-1)(n-1)}I^r.
\end{array}$$

(2) We have the containment $I^{(n(r-1))} \subseteq I^r$ for $r\gg0$.\\Let $k = nm+t$. From \cite{J}, we have $I^{(ns+a_1+\cdots+a_s)} \subseteq I^{(a_1+1)}I^{(a_2+1)}\cdots I^{(a_s+1)},$ letting $s = n+t, a_1 = a_2 = \dots = a_n = nm-n-1, a_{n+1} = \dots = a_{n+t} = 0$.

Let $k = nm + t = (v-1)n + t$ for $t \geq 0$ and let $s = n+t$ and 
$a_1 = a_2 = \cdots a_n = nm-n-1 = n(v-1)-n-1, a_{n+1} = \dots = a_t = 0$. Therefore, 
$$I^{(n(n+t)+ n(nm-n-1))} = I^{(n^2+nt +n^2m-n^2-n)} = I^{(nk-n)} \subseteq (I^{((n(v-1)-n))})^nI^t = I^{n(v-1)}I^t = I^k.$$

Hence $I^{(nk-n)} \subseteq I^k$
for $k\gg 0$. 
\end{proof}

As a consequence, we can show that the ideal of a Complement of a Steiner Configuration of points has expected resurgence, that is, its resurgence is strictly less than its big height (see \cite{GrHuM}).
\begin{Corollary}\label{expected}
    Let $I \subseteq k[x_0,\dots,x_n]$ be the ideal defining the Complement of a Steiner Configuration of points in $\mathbb{P}^n_k$. Then, $\rho(I) < n.$ 
\end{Corollary}

\begin{proof}
   From Theorem \ref{t. Stable main}, we have $\rho_a(I) < n$. Note that $\rho_a(I) \leq \rho(I)$. If $\rho_a(I) = \rho(I)$, then clearly $\rho(I) < n$. On the other hand, if $\rho_a(I) < \rho(I)$, then from Lemma \ref{l. main lemma} and Corollary \ref{expres}, we conclude that $\rho(I) < n.$
\end{proof}

We give an alternative proof of Chudnovky’s Conjecture:
\begin{Corollary}\label{Chudnovky}
 Let $I \subseteq k[x_0,\dots,x_n]$ be the ideal defining the Complement of a Steiner Configuration of points in $\mathbb{P}^n_k$. Then Chudnovky’s Conjecture holds for $I$. 
\end{Corollary}
\begin{proof}
 From Theorem \ref{p. alpha(J_T^m)}, item i) $\alpha(I)=v-n$ and from  Theorem \ref{coro2} it is $ \widehat{\alpha}(I)=\frac{v}{n}$. Then $$\widehat{\alpha}(I)\geq \frac{\alpha(I) + n -1}{n} \Leftrightarrow \frac{v}{n}\geq \frac{v-1}{n}.$$
\end{proof}

\begin{Corollary}\label{demailly}
 Let $I \subseteq k[x_0,\dots,x_n]$ be the ideal defining the Complement of a Steiner Configuration of points in $\mathbb{P}^n_k$. Then Demailly’s Conjecture holds for $I$. 
\end{Corollary}
\begin{proof}
From Theorem \ref{p. alpha(J_T^m)}, for $m = pn + q$ and for $2 \leq q < n$ it is $\alpha(I^{(m)}) = pv + \alpha(I^{(q)}) = pv + v - n + q$.
From Corollary \ref{coro2}, we have that $\widehat{\alpha}(I)=\frac{v}{n}$. Hence, whenever
\begin{enumerate}
    \item $m = pn + q$, with $2\leq  q < n$, 
    we have $$\frac{\alpha(I^{(m)}) + n-1}{m+n-1} = \frac{pv + v-n+q +n-1}{pn+q+n-1} = \frac{(p+1)v+q-1}{(p+1)n+q-1} \leq \frac{v}{n} = \widehat{\alpha}(I)$$
\item $q=1$ and $m = np +1$, we have $$\frac{\alpha(I^{(m)}) + n-1}{m+n-1} = \frac{pv + v-n +n-1}{pn+1+n-1} = \frac{(p+1)v-1}{(p+1)n} < \frac{v}{n} =\widehat{\alpha}(I)$$

\item $q=0$ and $m = np $, we have $$\frac{\alpha(I^{(m)}) + n-1}{m+n-1} = \frac{pv  +n-1}{pn+n-1} < \frac{v}{n} =\widehat{\alpha}(I).$$
\end{enumerate}
\end{proof}

\begin{Remark}
Chudnovsky's Conjecture can be showed from Proposition \ref{prop}. We have that
$I^{(nr)}\subseteq \mathcal{M}^{rn}I^r$.
This gives us the inequality $\alpha(I^{(nr)})\geq rn + r\alpha(I)\geq rn + r\alpha(I) - r$. Dividing both sides by
$nr$ and letting $r \rightarrow \infty$ gives
$$\widehat{\alpha}(I)\geq \frac{\alpha(I) + n -1}{n}.$$
\end{Remark}

\section{Containment and colouring}\label{s.coloring}

In this section we focus on the relation between the colourability of a hypergraph ${H}$ and the failure of the containment problem for the cover ideal associated to $H$. Then, we apply these results in the case that $H$ is a Steiner System. There exists an extensive literature on the subject of colourings both from Design Theory and Algebraic Geometry/Commutative Algebra point of view. Among all, we make use of  \cite{BTV, FHVT2, HVT, GMV} as some of referring texts. 

Most of the existing papers are devoted to the case of weak
colourings (or vertex colourings), i.e. colourings where the colours are assigned to the elements in such a way that no hyperedge is monochromatic (i.e. no hyperedge has all its elements assigned the same colour). The reader can see \cite{BTV} or  Chapter 3 in \cite{GMV} for other types of colouring a hypergraph, such as {\em  strong vertex colouring, vertex equicolouring, good colouring of $H$}.

In this paper we use the case of  weak colouring to get results on Containment problem.

We first recall some known definitions and results from \cite{BTV} or \cite{GMV}, Chapter 2.

    A hypergraph is a pair $H=(V,E)$, where $E=\{x_1,\dots,x_n\}$ is a finite nonempty set containing $n$ elements called vertices and $E=\{e_i\}_{i\in I}$ ($I$ set of indices) is a family of subsets of $X$, called edges, or otherwise hyperedges, such that for all $e\in E, e\neq \emptyset$ and $\cup_{e\in E}^{}\; e=X$.

 A colouring of a hypergraph $H= (V,E)$  is a surjective mapping $c: V \rightarrow C$ where $C$ is the
set of colours. When $|C|=m$, then a proper $m$-colouring of a hypergraph $H= (V,E)$  is a mapping $c:V\rightarrow \{1,2,\dots, m\}$ for which every edge $e\in E$ has at least two vertices of different colours.

As for graphs, proper colourings generate partitions of the vertex set into a number of stable (independent) non-empty subsets called colour classes, with as many classes as the number of colours actually used. 

Thus, we use an equivalent definition from \cite{FHVT2, HVT}, used in Combinatorial Algebraic Geometry/Commutative Algebra research, i.e, 
\begin{Definition}
Let $H= (V,E)$ be a hypergraph. An \textit{$m$-colouring} of $H$ is any partition of $V=U_1\cup \cdots \cup U_m$ into $m$ disjoint sets such that for every  $e\in E$ we have $e\not\subseteq U_j$ for all $j=1,\dots,m$.
The $U_j$'s are called the \textit{colour classes}.

 The  chromatic number of $H$, denoted by $\chi(H)$, is the minimum $m$ such that $H$ has an $m$-colouring.
 \end{Definition}

\begin{Definition}
	A hypergraph $H=(V,E)$ is  \textit{$m$-colourable} if there exists a proper $m$-colouring, i.e, if $\chi(H)\leq m$.
%We say $H$ is \textit{$m$-chromatic} if it is $m$-colorable but not $(m-1)$-colorable.
\end{Definition}

\begin{Definition}
We say $H$ is \textit{$m$-chromatic} if it is $m$-colourable but not $(m-1)$-colourable.
\end{Definition}

 When $\chi(H)\leq 2$, the hypergraph $H$ is called {\em bicolourable}. (In parts of the literature the term ‘bipartite’ is also used.) 

\begin{Definition}\label{d. d-cover}
	Let $H:=(V,E)$ be a hypergraph. For an integer $c,$ we say that $H$ is \textit{$c$-coverable} if there exists a partition $U_1, U_2, \dots, U_c$ of $V$ such that $e\cap U_i\neq \emptyset$ for each $i=1,\ldots,c$ and for each $e\in E.$ 
\end{Definition}

\begin{Remark}\label{r. cov col}
	Note that, as an immediate consequence of the above definitions,  if $H$ is $c$-coverable, $c>1$, then $H$ is $c$-colourable. 
\end{Remark}

\begin{Example} Set $V:=\{x_1,x_2,x_3,x_4,x_5,x_6,x_7\}$. Let $H$ be the set of blocks of a $STS(7)$
	
	\begin{equation}\label{fano}
	  H:=\{\{x_1, x_2, x_3\},\{x_1, x_4, x_5\},\{ x_1, x_6, x_7\},\{ x_2, x_4, x_6\},\{x_2, x_5, x_7\},\{x_3, x_4, x_7\},\{x_3, x_5, x_6\}\}.
	\end{equation} 
 Take, for instance, the partition $\{x_1,x_2,x_5\},\{x_3,x_4,x_6\},\{x_7\}$ (see Figure \ref{fig. steiner7}). 	$H$ is 3-colourable but it is not 3-coverable. 
	
Notice also that no colouring of $H$ with two colours exists. Then $\chi(H)=3$.	
	\begin{figure}[ht]
	\centering
	\includegraphics[scale=0.2]{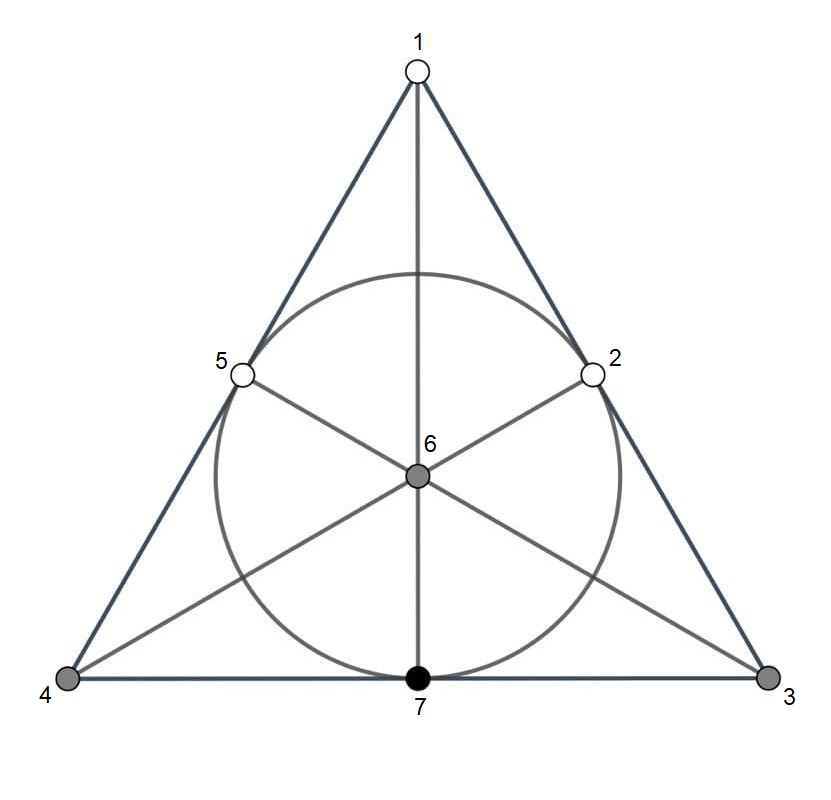}
 \caption{ The three colour classes of a $STS(7)$. }\label{fig. steiner7}

\end{figure}

\end{Example}

\begin{Remark}
We refer the reader to Section 3.5 in \cite{GMV} to see examples of different types of colourings that give different chromatic numbers for the same $H$. In particular, in  Example 8 of \cite{GMV}, the {\em strong} vertex colouring of $H$ as in (\ref{fano}) gives $\chi(H)=7$ (recall that a mapping $c$ is a strong colouring of vertices of $H$ if for all $e\in E$ it is $|c(e)|=|e|$).
\end{Remark}

For a non empty hypergraph $H$, i.e.   $H\subset 2^V$,  we define the ideal 
$$J_H:= \bigcap_{\sigma\in H} \mathfrak p_{\sigma} \subseteq k[V]$$   
called the {\em  cover ideal} of $H,$ 
where for a subset of $V$, $\sigma:=\{{i_1}, {i_2}, \ldots, {i_n}\}\subseteq V,$ 
the ideal $$\mathfrak p_{\sigma}:=(x_{i_1}, x_{i_2}, \ldots, x_{i_n})\subseteq k[V]$$
denotes the prime ideal generated by the variables indexed by $\sigma$.

For a hypergraph  $H=(V,B)$, we denote by $\tau(H):=\min_{b\in B}\{ |b| \}.$

We study some properties of the cover ideals of $B$. The following results show a relation between the coverability of a hypergraph $H$ and that the Containment problem can fail.

\begin{Theorem}\label{t.containement fails} Let $H=(V,B)$ be a hypergraph. If $H$ is not $d$-coverable then~$J_H^{(\tau(H))}\not \subseteq J_H^{d}.$
\end{Theorem}
\begin{proof}
	We put $\tau:=\tau(H)$ and $w:=x_1\cdot x_2\cdots x_n$. In order to prove the statement it is enough to show that $w\in J_H^{(\tau)}$ but $w\notin J_H^{d}.$
	For each $b\in B$ the ideal $\mathfrak p_b$ has height $|b|\le \tau$. Therefore $x_b\in (\mathfrak p_b)^{\tau}$. Thus $w \in \mathfrak p_b^{\tau}$ for each $b\in B$. This implies $w\in J_H^{(\tau)}.$
	By contradiction, assume $w\in J_H^{d}.$ Thus, there exist $w_1, \ldots, w_d\in J_H$ such that $w=w_1\cdots w_d.$ 
	We $U_j:=\{x_u\in V\ |\ x_u\ \text{divides}\ w_j \}$, then $U_1,\ldots, U_{d}$ is a partition  of $V$.
	Thus,  for each $b\in B$ we have $w_i\in \mathfrak p_b$, therefore $U_i\cap e \neq 0$ with $i=1,\ldots, d.$
	This contradicts that $H$ is not $d$-coverable.
\end{proof}

Recall that an $m$-colouring of $(V,B)$ is called an $m$-bicolouring if the vertices of each $b\in B$ are coloured with exactly two colours. A Steiner Triple Systems $(V,B)$ admitting an $m$-bicolouring is $m$-bicolourable.
Thus, in a bicolouring of a Steiner Triple System $(V,B)$, every triple has two elements in one colour class and one in another class, so there are no monochromatic triples nor polychromatic triples (i.e. triples receiving three colours).  For instance, for a deep investigation of colouring properties of Steiner Triple Systems the reader can see \cite{CDR}.

As a consequence, we get a failure of the containment for the cover ideals associated to Steiner Triple Systems $(V,B)$ of type $S(t,n,v).$

\begin{Proposition}
    If $v>3$ and $S(2,3,v)=(V,B)$ is a Steiner Triple System, then 
	$J_B^{(3)}\not\subseteq J_B^2$.
\end{Proposition}
\begin{proof}
	It is enough to show that $B$ in not 2-coverable. Assume by contradiction $V=U_1\cup U_2.$ By definition, for each $\{i,j,k\}\in B$ we have $\{i,j,k\}\cap U_1\neq \emptyset$ and $\{i,j,k\}\cap U_2\neq \emptyset$.  
	This implies that $S(2,3,v)$ is 2-bicolourable contradicting a well known fact about Steiner Triple Systems, see~\cite{R}.
\end{proof}

We end the paper showing the failure of the containment for the cover ideals associated to Steiner Systems.

\begin{Proposition}\label{p. t-coverability}
	Let $\mathcal S=(V,B)$ be a Steiner System with parameters $S(t,t+a,v)$  where $1\le a\leq  t-2$ and  $v>(a+1)t$.  Then,  $J_B^{(t+a)}\not \subseteq J_B^{t}.$
\end{Proposition}
\begin{proof}
	Note that, from Theorem \ref{t.containement fails},  it is enough to show that 	
	$B$ is not $t$-coverable.
	Assume by contradiction there is a partition of $V$ in  $t$ colour classes,  $V=C_1\cup\cdots\cup C_{t}$ such that $B\cap C_j\neq \emptyset$ for each $j=1,\ldots,t$ and $B$ a block of $\mathcal S.$ 
	We denote by $c_j$ the number of elements in $C_j$  for $j=1,\ldots,t$.
	Note that $c_j\le a+1$. Indeed if $i_1,i_2, \ldots, i_{a+2}\in C_j$ are different elements, then a block containing $i_1,i_2, \ldots, i_{a+2}$ cannot intersect $t$ colour classes.
 	This implies $v\le (a+1)t$.
\end{proof}

The next example shows that Theorem \ref{t.containement fails} does not characterize the failure of the containment. 
\begin{Example} Let $B$ denote the blocks of a Steiner quadruple system $SQS({8})=S(3,4,8)$ on the vertex set $V=\{x_1,x_2, x_3,x_4, x_5, x_6,x_7, x_8\}$,
	\[\begin{array}{rrl}
	B:=\{& \{x_1,x_2,x_3,x_4\}, \{x_1,x_2,x_5,x_6\}, \{x_1,x_2,x_7,x_8\}, \{x_1,x_3,x_5,x_7\}, \{x_1,x_3,x_6,x_8\},&\\
	   &\{x_1,x_4,x_5,x_8\}, \{x_1,x_4,x_6,x_7\}, \{x_2,x_3,x_5,x_8\}, \{x_2,x_3,x_6,x_7\}, \{x_2,x_4,x_5,x_7\},&\\
	  &\{x_2,x_4,x_6,x_8\}, \{x_3,x_4,x_5,x_6\}, \{x_3,x_4,x_7,x_8\}, \{x_5,x_6,x_7,x_8\}&\}.
	\end{array}\]

\noindent From Proposition \ref{p. t-coverability}, $B$ is not 3-coverable.
	Therefore Theorem \ref{t.containement fails} ensures $J_B^{(4)}\not\subseteq J_B^3.$
	However,  one can check that, for instance,
	$x_1x_2\cdots x_7\in J_B^{(3)}\setminus  J_B^2,$
	so the failure of the containment $J_B^{(3)}\subseteq  J_B^2$ cannot be motivated from Theorem \ref{t.containement fails}. 
	\end{Example}

\section{Conclusions}
Several conjectures have been posed on the Containment problem, creating an active area of current interests and ongoing investigations. In this paper, we show that the Stable Harbourne Conjecture and the Stable Harbourne -- Huneke Conjecture hold for the defining ideal of a Complement of a Steiner configuration of points in  $\mathbb{P}^{n}_{k}$.  Moreover, given a hypergraph $H$, we also study the relation between  its colourability  and that the Containment problem can fail for the cover ideal associated to $H$.  We wish to continue the study of  Steiner configurations of points and their Complements, since they are special subsets of star configurations whose Hilbert Function is the same as  sets of generic points while geometrically they are far of being generic. 

 We end this section recalling some open questions that are still under investigations and posing  new ones.
 
 We recall from \cite{BFGM}, that from a combinatorial point of view, two Steiner systems having the same parameters could have very different properties and such differences effect the homological invariants. Using experiments with  \cite{Cocoa} and \cite{Macaulay2} we ask:
	
	\begin{Question}\label{Q1}
		Let $(V,B)$ be a Steiner system of type $S(t,n,v)$, and  $X_{\mathcal H,B}$ the associated Steiner configuration of points.  Assume that the hyperplanes in $\mathcal H$ are chosen generically.  Do the Hilbert function and the graded Betti numbers of $X_{\mathcal H,B}$ only depend on $t,n,v$? 
	\end{Question}
	
	\begin{Question}\label{Q2}
		Let $(V,B)$ be a Steiner system of type $S(t,n,v)$, and  $X_{\mathcal H,B}$ the associated Steiner configuration of points. Assume that the hyperplanes in $\mathcal H$ are chosen generically.  Are the Hilbert function and the graded Betti numbers of $X_{\mathcal H,B}$ \textit{generic with respect to the Hilbert function?} (i.e. the same as a set of $|B|$ generic points in $\mathbb P^n$?) 
	\end{Question}
	
Given a hypergraph ${H}$, we also study the relation between  its colourability  and the failure of the containment problem for the cover ideal associated to $H$.  We suggest the following
	\begin{Question}
	 Can different types of colourings of a hypergraph give different answers to the Containment problem and  related conjectures?
	\end{Question}
	
	%In \cite{BFGM}, the authors apply some results to Coding Theory. 
	We also thank one of the referees to point out \cite{Ja1, Ja2}, where the author studies graph partitioning (fragmentation criteria) that  has   many   fields   of   applications   in   engineering,  especially  in  applied  sciences  as  applied  chemistry  and  physics,  computer science and automation, electronics and telecommunication. See \cite{Bal, Diu, Ran} just to cite some of them.
	
	%The author uses lexicographic ordering as as one of the %criteria for graph partitioning and related colouring.
\begin{Question}
%Can different ordering affect the graph partitioning?  Or 
Can different types of colourings of hypergraphs give also different answers to fragmentation criteria?
\end{Question}

%%%%%%%%%%%%%%%%%%%%%%%%%%%%%%%%%%%%%%%%%%
%\authorcontributions{All the authors gave the same contribution in preparing the first draft of the paper.  All authors have %read and agreed to the published version of the manuscript, except L. Milazzo since he died in 2019. But all the authors %recognize his contribution in the first draft of the paper.}

%%%%%%%%%%%%%%%%%%%%%%%%%%%%%%%%%%%%%%%%%%
%\funding{This research received no external funding}

%%%%%%%%%%%%%%%%%%%%%%%%%%%%%%%%%%%%%%%%%%
\begin{Acknowledgments}This paper represents the second part of a project started in April 2018 when E. Guardo and  L. Milazzo met together with the idea of making an interdisciplinary project between their own fields of research, such as Combinatorics and Commmutative Algebra/Algebraic Geometry. L. Milazzo died in 2019 but all the authors recognize his contribution in Section \ref{s.coloring}. We thank M. Gionfriddo and S. Milici for encouraging us to continue this project. We also thank Zs. Tuza for his useful comments in the revised version of the paper.
 Favacchio and Guardo's work has  been supported by the Universit\`{a} degli Studi di Catania, "Piano della Ricerca 2020/2022 Linea di intervento 2". The first three authors are members of GNSAGA of INdAM (Italy).
The computer softwares CoCoA \cite{Cocoa} and  Macaulay2 \cite{Macaulay2} were indispensable for all the computations we made.

The authors thank  all the referees for the useful suggestions/comments that improved the first version of the paper.
\end{Acknowledgments}


\begin{thebibliography}{999}
% Reference 1
%\bibitem[Author1(year)]{ref-journal}
%Author1, T. The title of the cited article. {\em Journal %Abbreviation} {\bf 2008}, {\em 10}, 142--149.
% Reference 2
%\bibitem[Author2(year)]{ref-book}
%Author2, L. The title of the cited contribution. In {\em The Book %Title}; Editor1, F., Editor2, A., Eds.; Publishing House: City, %Country, 2007; pp. 32--58.
%\end{thebibliography}



%\begin{thebibliography}{99}
\bibitem{Cocoa}  Abbott J.,  Bigatti A. M.,  Robbiano L.
\textit{CoCoA: a system for doing Computations in Commutative Algebra.} Available at http://cocoa.dima.unige.it

\bibitem{Bal} Balaban  A.\, T.,  Solved  and  Unsolved  Problems  in  Chemical  Graph  Theory, {\em Ann. Discrete Math.} {\bf 1993} {\em 55}, 109--126.

\bibitem{BFGM} Ballico E., Favacchio G., Guardo E., Milazzo L. {Steiner systems and configurations of points.} {\em Des. Codes Cryptogr.} {\bf 2020}, https://doi.org/10.1007/s10623-020-00815-x.
	
	
\bibitem{BDHKaKnSySz} Bauer T., Di Rocco S., Harbourne B., Kapustka M., Knutsen A., Syzdek W., Szemberg T. {A primer on Seshadri Constants,} In {\em Interactions of Classical and Numerical Algebraic Geometry}; Contemp. Math., 496; 2009; pp 33--70.

\bibitem{BiGrTN} Bisui S., Grifo E., H\'a H. T., Nguy\^en T.T. {Chudnovsky's Conjecture and the stable Harbourne--Huneke containment} {\em Preprint 2020 https://arxiv.org/abs/2004.11213}

\bibitem{BiGrTN2} Bisui S., Grifo E., H\'a H. T., Nguy\^en T.T. {Demailly's Conjecture and the Containment problem} {\em Preprint 2020 https://arxiv.org/pdf/2009.05022.pdf}

\bibitem{BCH} Bocci C., Cooper S.,  Harbourne B. {Containment results for ideals of variuos configurations of points in $\mathbb{P}^N$}, {\em J. Pure Appl. Algebra} {\bf 2014}, {\em 218 (1)},  65--75.
	
\bibitem{BH2} Bocci C.,  Harbourne B. {Comparing powers and symbolic powers of ideals.} {\em J. Algebraic Geom.} {\bf 2010}, {\em 19 (3)},  399--417.

\bibitem{BH} Bocci C., Harbourne B. {The resurgence of ideal of points and the containment problem}. {\em Proc. Amer. Math. Soc.} {\bf 2010}, {\em 138},  1175--1190.

\bibitem{BTV} Bujtas Cs., Tuza Zs., Voloshin V. {Hypergraphs colouring.} In {\em Topics in chromatic graph theory,  Chapter 11}; L. W. Beineke, R. J. Wilson; Cambridge University press, 2015; pp 230--254.

\bibitem{CCGVT} Carlini E., Catalisano M. V., Guardo E., Van Tuyl A., { Hadamard Star Configurations}, {\em Rocky Mountain J. Math.} {\bf 2019},{\em 49 (2)},  419--432.

\bibitem{CGVT} Carlini E., Guardo E., Van Tuyl A. {Star configurations on generic hypersurfaces}, {\em J. Algebra} {\bf 2014}, {\em 407},  1--20.

\bibitem{CHHVT1} Carlini E., H\'a H. T., Harbourne B., Van Tuyl A. {Ideals of Powers and Powers of Ideals.} In {\em Lecture Notes of the Unione Matematica Italiana}; Springer, Cham., 2020; pp. 1--161

\bibitem{Colbourn} Colbourn C.J., Dinitz J. H. Handbook of combinatorial designs. In  {\em Handbook of combinatorial designs}, Second Edition, CRC press; 2006, pp. 1--984.

\bibitem{CDR} Colbourn C.J., Dinitz J. H., Rosa A. {Bicoloring Steiner Triple Systems}, {\em Electron. J. Combin.} {\bf 1999}, {\em 6}, Article Number R25, 1--16.
	
\bibitem{ColbournRosa} Colbourn C.J., Rosa A. {\em Triple systems}.  Oxford University Press, 1999.

\bibitem{CEH} Cooper S. M., Embree R. J., H\'a H. T., Hoefel A. H.  {Symbolic powers of monomial ideals.} {\em Proc. Edinb. Math. Soc.(2)} {\bf 2017}, {\em 60 (1)},  39--55.	

	
\bibitem{D} Denkert A. {Results on containments and resurgences, with a focus on ideals of points in the plane}. {\em PhD Thesis} {\bf 2013}.

\bibitem{DD} DiPasquale M., Drabkin B. {On Resurgence via Asymptotic Resurgence} {\em arXiv:2003.06980}.
		
  \bibitem{DiPFMS} DiPasquale M., Francisco C., Mermin J., Schweig J. {Asymptotic Resurgence via integral closures} {\em Trans. Amer. Math. Soc.} {\bf 2019}, {\em 372},  6655--6676.
  
  \bibitem{Diu} Diudea  M.  V.,  Topan  M.,  Graovac  A., Layer  Matrices  of  Walk  Degrees, {\em J.  Chem.  Inf.  Comput. Sci.} {\bf 1994}, {\em 34},  1072--1078.
  
 \bibitem {DHN+} Dumnicki M., Harbourne B., Nagel U., Seceleanu A., Szemberg T., Tutaj-Gasi\'nska H. {Resurgences for ideals of special point configurations in $\mathbb{P}^N$ coming from hyperplane arrangements}. {\em J. Algebra} {\bf 2015}, {\em 443}, 383--394.
  
  	\bibitem{DST} Dumnicki M., Szemberg T., Tutaj-Gasi\'nska H. {Counterexamples to the containment $I^{(3)} \subseteq I^2$ containment}. {\em J. Algebra} {\bf 2013}, {\em 393}, 24--29.
  
 	\bibitem{ELS} Ein L., Lazarsfeld R., Smith K.E. {Uniform bounds and symbolic powers on smooth varieties}. {\em Invent. Math.} {\bf 2001}, {\em  144}, 241--252.
  
  \bibitem{FG} Favacchio G., Guardo E. {The Minimal Free Resolution of Fat Almost Complete Intersections in $\mathbb{P}^1\times \mathbb{P}^1$.} {\em Canad. J. Math.} {\bf 2017}, {\em  69 (6)}, 1274--1291.
  
  \bibitem{FG2} Favacchio G., Guardo E. {On the Betti numbers of three fat points in $\mathbb{P}^1\times \mathbb{P}^1$} , {\em J.  Korean Math. Soc.} {\bf 2019}, {\em 56 (3)},  751--766, doi.org/10.4134/JKMS.j180385
  
  \bibitem{FMX} Fouli L., Mantero P., Xie Y. {Chudnovsky’s conjecture for very general points in $\mathbb{P}^n$}. {\em J. Algebra} {\bf 2018}, {\em 498}, 211--227.


\bibitem{FHVT2} Francisco C.,  H\'a H. T., Van Tuyl A. {Colorings of hypergraphs, perfect graphs, and associated primes of powers of monomial ideals}, {\em J. Algebra} {\bf 2011}, {\em 331 (1)},  224--242.
  
	\bibitem{GHM} Geramita A., Harbourne B., Migliore J. {Star configurations in $\mathbb P^n$}. {\em J. Algebra} {\bf 2013}, {\em 376},  279--299.
	
	\bibitem{GHMN} Geramita A., Harbourne B., Migliore J., Nagel U. {Matroid configurations and symbolic powers of their ideals}. {\em Trans. Amer. Math. Soc.} {\bf 2017}, {\em 369 (10)},  7049--7066.	
	
\bibitem{GMV} Gionfriddo M., Milazzo L., Voloshin V. {Hypergraphs and Designs} In {\em Nova Science Pub Inc}, New York, NY, USA, { 2015;} pp. 1--169.
	
	\bibitem{Macaulay2} Grayson D.R, Stillman M. E. \textit{Macaulay 2, a software system for research in algebraic geometry}.

	\bibitem{Gr} Grifo E. {A stable version of Harbourne's Conjecture and the containment problem for space monomial curves}, {\em J. Pure  Appl. Algebra} {\bf 2020}, {\em 224 (12)}, Article 106435.

\bibitem{GrHuM} Grifo E., Huneke C., Mukundan V. {Expected Resurgence and symbolic powers of Ideals}, {\em J. Lond. Math. Soc.} {\bf 2020}, {\em 102}, 453--469.
	
	\bibitem{GHaVT} Guardo E., Harbourne B., Van Tuyl A. {Asymptotic resurgence for ideals of positive dimensional subschemes in projective space}, {\em Adv. Math.} {\bf 2013}, {\em 246},  114--127.

\bibitem{GHaVT1} Guardo E., Harbourne B., Van Tuyl A. {Symbolic powers versus regular powers of ideals of general points in $\mathbb P^1 \times \mathbb P^1$}, {\em Canad. J. Math.} {\bf 2013}, {\em  65 (4)},  823--842.

\bibitem{GHaVT2013a} Guardo E., Harbourne B., Van Tuyl A. {Fat lines in $\mathbb P^3$: powers versus symbolic powers}, {\em J. Algebra} {\bf 2013}, {\em 390}, 221--230.

\bibitem{HVT}  H\'a  H.T., Van Tuyl A. {Monomial ideals, edge ideals of hypergraphs, and their minimal graded free resolutions}, {\em J. Algebraic Combin.} {\bf 2008}, {\em 27 (2)}, 215--245.

\bibitem{HaHu} Harbourne B., Huneke C.{Are symbolic powers highly evolved?} {\em J. Ramanujan Math. Soc.} {\bf 2013}, {\em 28 (3)}, 311--330.
\bibitem{HKZ} Harbourne B., Kettinger J., Zimmitti F. {Extreme values of the resurgence for homgeneous ideals in polynomials rings},{\em https://arxiv.org/pdf/2005.05282.pdf}.
	
   \bibitem{HoHu} Hochster M., Huneke C. {Comparison of Symbolic and ordinary powers of ideals}, {\em Invent. Math.} {\bf 2002}, {\em 147}, 349--369.
   
 \bibitem{Ja1} J\"antschi L. {Graph Theory. 1. Fragmentation of Structural Graphs.} {\em Leonardo EI J. Pract. Technol.}  {\bf 2002}, {\em 1(1)}, 19--36.
 
 \bibitem{Ja2} J\"antschi L. {Graph Theory. 2. Vertex Descriptors and Graph Coloring.} {\em Leonardo EI J. Pract. Technol.}  {\bf 2002}, {\em 1(1)}, 37--52.
 
\bibitem{J} Johnson R. {Containing Symbolic powers in Regular Rings}, {\em Comm.  Algebra} {\bf 2014}, {\em 42}, 3552--3557.

\bibitem{LM} Lin K.N., McCullough J. {Hypergraphs and regularity of square-free monomial ideals}. {\em Internat. J.  Algebra  Comput.} {\bf 2013}, {\em 23 (7)}, 1573--1590. 
	
	\bibitem{MS} Ma L., Schwede K. {Perfectoid multiplier test ideals in regular rings and bounds on symbolic powers}. {\em Invent. Math.} {\bf 2018}, {\em 214}, 913--955.
	
	\bibitem{Mantero} Mantero P. The structure and free resolutions of the symbolic powers of star configurations
of hypersurfaces. {\bf 2019} {\em arXiv:1907.08172}.

	\bibitem{N} Nagata M. \textit{Local rings}. Interscience, 1962; pp. 1-- 234.

\bibitem{Ran}	Randi\'c  M.,  Design  of  Molecules  with  Desired  Properties  (A  Molecular  Similarity  Approach   to   Property   Optimization), In {\em Concept   and   Applications   of   Molecular   Similarity}, Ed. A. M. Johnson, G. M. Maggiora, John Wiley \& Sons, 5 ,  { 1990}, pp. 77--145.

	\bibitem{R} Rosa A. Steiner triple systems and their chromatic number, {\em Acta Fac.	Rer. Nat. Univer. Comen. Math.} {\bf 1970}, {\em 24}, 159--174.
	
	\bibitem{S} Swanson I. {Linear equivalence of ideal topologies} {\em Math. Z.} {\bf 2000}, {\em 234},  755--775.

	\bibitem{SH} Swanson I., Huneke C. {Integral Closure of Ideals, Rings and Modules}, In {\em London Math. Soc. Lecture Note Ser. 336}, Cambridge University Press, { 2006}; pp. 1--463.

	\bibitem{W} Waldschmidt M. \textit{Propri\'et\'es arithm\'etiques de fonctions de plusieurs variables. II.}
In {\em S\'eminaire P.\ Lelong (Analyse), 1975/76}, Lecture Notes Math. 578, Springer, { 1977};  pp. 108--135.


\bibitem{Z} Zariski O. {A fundamental lemma from the theory of holomorphic functions on an algebraic variety}, {\em Ann. Mat. Pura Appl.} {\bf 1949}, {\em (4),29},  187-–198.

	
\end{thebibliography}
\end{document}